\pdfoutput=1
\documentclass[11pt,a4paper]{article}
\usepackage[margin=2.6cm]{geometry}
\usepackage{amsmath,amssymb,amsthm,mathtools}
\usepackage{booktabs}
\usepackage[expansion=false]{microtype}
\usepackage{xcolor}
\usepackage[colorlinks=true,linkcolor=blue!55!black,citecolor=blue!55!black,urlcolor=blue!55!black]{hyperref}

\theoremstyle{plain}
\newtheorem{theorem}{Theorem}[section]
\newtheorem{proposition}[theorem]{Proposition}
\newtheorem{lemma}[theorem]{Lemma}
\newtheorem{corollary}[theorem]{Corollary}
\theoremstyle{definition}
\newtheorem{remark}[theorem]{Remark}
\newtheorem{question}[theorem]{Question}

\DeclarePairedDelimiter{\abs}{\lvert}{\rvert}
\newcommand{\R}{\mathbb{R}}
\newcommand{\Q}{\mathbb{Q}}
\newcommand{\N}{\mathbb{N}}
\newcommand{\dimH}{\dim_{\mathcal H}}
\newcommand{\normx}[1]{\lVert #1 \rVert}
\newcommand{\Fst}{\mathcal{F}_{s,t}}
\newcommand{\Gst}{\mathcal{G}_{s,t}}

\title{\Large Critical convergence and Hausdorff measures\\
for generalized Flint Hills series}
\author{Yuya Dan\thanks{Danlab, Faculty of Informatics, Matsuyama University.
\texttt{dan@g.matsuyama-u.ac.jp}}}
\date{\today}

\begin{document}
\maketitle

\begin{abstract}
We study the generalized Flint Hills series
\[
  \Fst(x)=\sum_{n\ge1}n^{-s}\abs{\sin(\pi nx)}^{-t}
\]
for $s>0$ and $t>1$. An explicit comparison with a series over continued-fraction
denominators yields the Hausdorff dimension $\min\{1,2t/(s+t)\}$ of its divergence set.
At each critical exponent $\tau=1+s/t>2$ we construct numbers of irrationality exponent
$\tau$ realizing both convergence and divergence; the convergent examples establish
Meiburg's conjecture in the range $t>1$. Both parts of the critical fibre have Hausdorff
dimension $2/\tau$. For $s>t$ and
\[
  h_\kappa(r)=r^{2/\tau}\bigl(\log(1/r)\bigr)^{\kappa}
\]
we prove that the divergence set has zero $h_\kappa$-measure for $\kappa<-1$ and
infinite measure for $\kappa\ge-1$. The divergent part of the critical fibre satisfies
the same law, whereas its convergent part has infinite measure for every $\kappa$. The
key estimate selects a rapidly growing subsequence of convergent denominators and gives
a double-logarithmic bound on the approximation error. The convergence of the classical
Flint Hills series $\sum_{n\ge1}(n^{3}\sin^{2}n)^{-1}$ remains undecided.


\end{abstract}

\section{Introduction}

The Flint Hills series
\begin{equation}\label{eq:FH}
  S \;=\; \sum_{n=1}^{\infty}\frac{1}{n^{3}\sin^{2}n}
\end{equation}
is not known to converge. The obstruction is entirely Diophantine: $\abs{\sin n}$ is
small exactly when $n$ is close to an integer multiple of $\pi$, so the size of the
terms of \eqref{eq:FH} is governed by the quality of the rational approximations to
$\pi$. Alekseyev \cite[Cor.~4]{Alekseyev} proved that convergence of \eqref{eq:FH}
implies $\mu(\pi)\le5/2$ for the irrationality exponent $\mu$, and gave the sufficient
condition $\mu(\pi)<1+(u-1)/v$ for $\sum_n n^{-u}\abs{\sin n}^{-v}$
\cite[Thm.~5]{Alekseyev}, which is vacuous at $(u,v)=(3,2)$. Meiburg \cite{Meiburg}
sharpened the sufficient condition to $\mu(\pi)<1+u/v$, so that $\mu(\pi)<5/2$ does
imply convergence of \eqref{eq:FH}. The best bound available today is
$\mu(\pi)\le7.103205334137\ldots$ \cite{ZeilbergerZudilin}, and the expected value is
$\mu(\pi)=2$, which would give convergence. The gap between $2.5$ and $7.10$ is
therefore the main obstacle to this approach, and closing it by the methods that produced
the known bounds \cite{Hata,Salikhov,ZeilbergerZudilin} would require a substantial new
idea; their recorded progress has been $8.0161\to7.6063\to7.1032$ over three decades.

The purpose of this note is not to attack that gap but to describe precisely what
\emph{is} decidable about \eqref{eq:FH} and its natural family. Five things are done.

\begin{enumerate}
\item Section~\ref{sec:red} records an explicit two-sided comparison
(Proposition~\ref{prop:red}) between $\Fst(x)$ and $\sum_k q_{k+1}^t/q_k^s$. This is
classical in substance --- it is the mechanism behind \cite{Alekseyev} and
\cite{Meiburg} --- but the two-sided form with explicit constants gives both the
necessary and the sufficient criterion at once, and every statement below follows from
it. Two unconditional corollaries: $\sum_n(n^{s}\sin^2n)^{-1}$ diverges for every
$s\le2$, and converges for every $s>2\mu(\pi)-2$, hence unconditionally for
$s>12.20642$.

\item Section~\ref{sec:dim} determines the Hausdorff dimension of the divergence set
of the family; for the Flint Hills exponents the answer is $4/5$
(Theorem~\ref{thm:dim}). The divergence set is null but of positive dimension and
dense, which is a quantitative form of the statement that a size estimate for
$D_{3,2}$ cannot decide whether the particular number $1/\pi$ belongs to it.

\item Section~\ref{sec:sharp} treats the critical case $\mu(x)=1+s/t$, left open by the
criterion. A single one-parameter family of continued fractions realizes both
behaviours there (Proposition~\ref{prop:sharp}). For $u>v>1$ the member $x_0$ recovers both
nonconvergence conclusions of \cite[Thm.~3.1]{Meiburg}, whereas the members $x_r$ with
$0<r\le1$ have divergent series whose terms tend to zero; the members with $r>1$
establish \cite[Conj.~3.1]{Meiburg} in this parameter range, and yield
Corollary~\ref{cor:noiff}: convergence is not equivalent to any condition on $\mu$
alone.

\item Section~\ref{sec:fibre} sets up the two parts $C_\tau,E_\tau$ of the critical
fibre and records the reformulation on which everything after it rests: writing
$\rho_k=q_k^{\tau}\abs{x-p_k/q_k}$, the criterion of Section~\ref{sec:red} becomes
$\Fst(x)<\infty\iff\sum_k\rho_k^{-t}<\infty$ (Lemma~\ref{lem:bridge}). This places the
problem in the setting of Bugeaud's theorem on sets of exact approximation order
\cite{BugeaudExact}, which gives $\dimH C_\tau=\dimH E_\tau=2/\tau$ at once
(Remark~\ref{rem:viaexact}); the same conclusion follows from Section~\ref{sec:gauge}
without that input (Corollary~\ref{thm:fibre}).

\item Section~\ref{sec:gauge} separates them by a gauge, and locates the threshold
exactly. The key point is Lemma~\ref{lem:Dcover}: after discarding the indices that
cannot contribute to $\sum_k\rho_k^{-t}$, the surviving denominators grow doubly
exponentially, so divergence forces
$\rho_k<\Lambda_\tau^{a}(\log\log q_k)^{a}$ infinitely often with $a=(1+\varepsilon)/t$;
the whole divergence set then sits inside a single limsup set to which the
Hausdorff--Cantelli lemma applies. With Jarn\'ik's theorem on the other side this gives
the zero--infinity law of Theorem~\ref{thm:gaugeD} and, for the fibre,
$\mathcal H^{h_\kappa}(E_\tau)=0<\infty=\mathcal H^{h_\kappa}(C_\tau)$ exactly for
$\kappa<-1$ (Theorem~\ref{thm:gaugefibre}). No Hausdorff-measure refinement of
Theorem~\ref{thm:bugeaud} is needed; the recent quantitative work on exact-order sets
\cite{BakerWard} computes dimensions only.
\end{enumerate}

Section~\ref{sec:num} collects numerical evidence about $\pi$ over a finite range of its
continued fraction, and Section~\ref{sec:disc} lists what seems reachable and what does
not. Approaches to \eqref{eq:FH} that avoid irrationality measures altogether have been
proposed, e.g.\ \cite{Agama}; we do not pursue them here.

\section{Notation and preliminaries}\label{sec:prelim}

Write $\normx{y}=\min_{m\in\mathbb Z}\abs{y-m}$ for the distance from $y\in\R$ to the
nearest integer. For $s>0$, $t>0$ and $x\in\R\setminus\Q$ set
\[
  \Fst(x)=\sum_{n=1}^{\infty}\frac{1}{n^{s}\,\abs{\sin(\pi n x)}^{t}},
  \qquad
  \Gst(x)=\sum_{n=1}^{\infty}\frac{1}{n^{s}\,\normx{nx}^{t}} .
\]
Both are unchanged by $x\mapsto x+1$, so we may and do assume $x\in(0,1)$ throughout,
i.e.\ $a_0=0$ in the continued fraction expansion below. Since
$\abs{\sin(\pi y)}=\sin(\pi\normx{y})$ and $\tfrac{2}{\pi}u\le\sin u\le u$ on
$[0,\pi/2]$, we have $2\normx{y}\le\abs{\sin(\pi y)}\le\pi\normx{y}$ for all $y$, whence
\begin{equation}\label{eq:FG}
  \pi^{-t}\,\Gst(x)\ \le\ \Fst(x)\ \le\ 2^{-t}\,\Gst(x).
\end{equation}
In particular $\Fst(x)$ and $\Gst(x)$ converge or diverge together. Taking $x=1/\pi$
gives $\sin(\pi n x)=\sin n$, so the Flint Hills series \eqref{eq:FH} is
$S=\mathcal F_{3,2}(1/\pi)$.

Let $x=[0;a_1,a_2,\dots]$ be the continued fraction expansion of an irrational
$x\in(0,1)$, with convergents $p_k/q_k$, so that $q_{-1}=0$, $q_0=1$, $p_0=0$ and
$q_{k+1}=a_{k+1}q_k+q_{k-1}$. We use the following standard facts
(see \cite[Ch.~I--II]{Khinchin}):
\begin{align}
  &\frac{1}{q_{k+1}+q_k}\;<\;\normx{q_kx}=\abs{q_kx-p_k}\;<\;\frac{1}{q_{k+1}}
    \qquad(k\ge1), \label{eq:cf1}\\[2pt]
  &0<n<q_{k+1}\ \Longrightarrow\ \normx{nx}\ \ge\ \normx{q_kx}
    \qquad(k\ge0), \label{eq:cf2}\\[2pt]
  &q_{k+1}\ge q_k+q_{k-1},\quad\text{hence}\quad q_k\ \ge\ \varphi^{\,k-1},\qquad
    \varphi=\tfrac{1+\sqrt5}{2}. \label{eq:cf3}
\end{align}

\begin{remark}\label{rem:k0}
The index restriction in \eqref{eq:cf1} is needed: for $k=0$ one has
$\normx{q_0x}=\normx{x}$ and $\abs{q_0x-p_0}=x$, which differ when $x>1/2$, i.e.\ when
$a_1=1$. If $a_1\ge2$ then $x<1/2$ and \eqref{eq:cf1} holds at $k=0$ as well, since
$\normx{x}=x>1/(a_1+1)=1/(q_1+q_0)$. Wherever \eqref{eq:cf1} is invoked at $k=0$ below,
$a_1\ge2$ holds, so this is the only case that occurs; see the proof of
Proposition~\ref{prop:red}.
\end{remark}

The \emph{irrationality exponent} of $x$ is
\[
  \mu(x)=\inf\Bigl\{\tau>0:\ \bigl|x-\tfrac pq\bigr|<q^{-\tau}
  \text{ for at most finitely many } \tfrac pq\in\Q\Bigr\},
\]
and, because the best rational approximations to $x$ are its convergents,
\eqref{eq:cf1} gives the continued fraction formula
\begin{equation}\label{eq:mu}
  \mu(x)\;=\;1+\limsup_{k\to\infty}\frac{\log q_{k+1}}{\log q_k}
  \;=\;2+\limsup_{k\to\infty}\frac{\log a_{k+1}}{\log q_k},
\end{equation}
the second equality because $q_k a_{k+1}\le q_{k+1}\le 2q_ka_{k+1}$.
By Khinchin's theorem $\mu(x)=2$ for Lebesgue-almost every $x$, and $\mu(x)\ge2$ for
every irrational $x$.

Since the Flint Hills series involves $1/\pi$ while the literature quotes $\mu(\pi)$,
we record:

\begin{lemma}\label{lem:inv}
If $x>0$ is irrational then $\mu(1/x)=\mu(x)$.
\end{lemma}

\begin{proof}
Let $0<\tau<\mu(x)$ and take infinitely many $p/q$ in lowest terms, $q\ge1$, with
$\abs{x-p/q}<q^{-\tau}$; necessarily $p\ge1$ for all large $q$ and $p/q\to x$. From
\[
  \Bigl|\frac1x-\frac qp\Bigr| \;=\; \frac{\abs{p-qx}}{p\,x}
  \;=\;\frac{q}{p\,x}\,\Bigl|x-\frac pq\Bigr|
\]
and $q/p\to 1/x$ we get $\abs{1/x-q/p}\le C_x\,q^{-\tau}$ along the same sequence, while
$p\le C_x' q$. Hence $\abs{1/x-q/p}\le C_x''\,p^{-\tau}\le p^{-\tau+\varepsilon}$ for
every $\varepsilon>0$ and all large $p$, and the fractions $q/p$ are pairwise distinct.
Therefore $\mu(1/x)\ge\tau-\varepsilon$ for all such $\tau,\varepsilon$, i.e.
$\mu(1/x)\ge\mu(x)$. Applying this to $1/x$ gives the reverse inequality.
\end{proof}

\section{The reduction to continued fraction denominators}\label{sec:red}

\begin{proposition}\label{prop:red}
Let $s>0$, $t>1$, let $x\in(0,1)$ be irrational with convergent denominators
$(q_k)_{k\ge0}$, and put
\[
  \Sigma_{s,t}(x)\;=\;\sum_{k\ge0}\frac{q_{k+1}^{\,t}}{q_k^{\,s}} .
\]
Then
\begin{equation}\label{eq:red}
  \sum_{k\ge1}\frac{q_{k+1}^{\,t}}{q_k^{\,s}}
  \;\le\;\Gst(x)\;\le\; C(s,t)\,\Sigma_{s,t}(x),
  \qquad
  C(s,t)=\frac{2^{s+t}\bigl(2+2^{t}\zeta(t)\bigr)}{1-2^{-s}} .
\end{equation}
In particular, by \eqref{eq:FG},
\[
  \Fst(x)<\infty \iff \Gst(x)<\infty \iff \Sigma_{s,t}(x)<\infty .
\]
\end{proposition}

\begin{proof}
\emph{Lower bound.} By \eqref{eq:cf1} the single term $n=q_k$, $k\ge1$, contributes
\[
  \frac{1}{q_k^{\,s}\normx{q_kx}^{t}}\;>\;\frac{q_{k+1}^{\,t}}{q_k^{\,s}} ,
\]
and for $k\ge1$ the integers $q_k$ are pairwise distinct, since
$q_{k+1}=a_{k+1}q_k+q_{k-1}>q_k$ for $k\ge1$. Summing over $k\ge1$ gives the left
inequality.

\emph{Upper bound.} Decompose $\N$ into dyadic blocks
$B_i=\{n\in\N:\,2^{i}\le n<2^{i+1}\}$, $i\ge0$, and let $k=k(i)$ be the largest index
with $q_k\le 2^{i+1}$; thus $q_{k+1}>2^{i+1}>n$ for every $n\in B_i$. Note that
$k(i)=0$ forces $q_1>2^{i+1}\ge2$, hence $a_1\ge3$, so \eqref{eq:cf1} is available for
$k=k(i)$ in every case, by Remark~\ref{rem:k0}. If $n\ne n'$ lie in $B_i$ then
$0<\abs{n-n'}<2^{i}<q_{k+1}$, so by \eqref{eq:cf2} and \eqref{eq:cf1},
\[
  \normx{(n-n')x}\ \ge\ \normx{q_kx}\ >\ \frac{1}{q_{k+1}+q_k}\ >\ \frac{1}{2q_{k+1}}
  \;=:\;\delta .
\]
Hence the points $\{nx\}$, $n\in B_i$, are $\delta$-separated in $\R/\mathbb Z$.
Consequently, for any $r>0$ the number of $n\in B_i$ with $\normx{nx}<r$ is at most
$2r/\delta+1$, because such points lie in an arc of length $2r$ around $0$. Order
$B_i=\{n_1,n_2,\dots\}$ so that $\normx{n_1x}\le\normx{n_2x}\le\cdots$. Taking
$r=(j-2)\delta/2$ for $j\ge3$ shows $\normx{n_jx}\ge (j-2)\delta/2$, while
$\normx{n_1x},\normx{n_2x}\ge\normx{q_kx}>\delta$ by \eqref{eq:cf2}. Therefore
\[
  \sum_{n\in B_i}\frac{1}{\normx{nx}^{t}}
  \;\le\;\frac{2}{\delta^{t}}+\sum_{j\ge3}\frac{2^{t}}{(j-2)^{t}\delta^{t}}
  \;=\;\frac{2+2^{t}\zeta(t)}{\delta^{t}}
  \;=\;2^{t}\bigl(2+2^{t}\zeta(t)\bigr)\,q_{k(i)+1}^{\,t},
\]
where $t>1$ was used for $\zeta(t)<\infty$. Since $n\ge2^{i}$ on $B_i$,
\[
  \Gst(x)\;\le\;\sum_{i\ge0}\frac{1}{2^{is}}\sum_{n\in B_i}\frac{1}{\normx{nx}^{t}}
  \;\le\;2^{t}\bigl(2+2^{t}\zeta(t)\bigr)\sum_{i\ge0}\frac{q_{k(i)+1}^{\,t}}{2^{is}} .
\]
Group the indices $i$ according to the value $k=k(i)$. Since $k(i)$ is non-decreasing,
$\{i:k(i)=k\}$ is a (possibly empty) interval of integers, and each of its elements
satisfies $2^{i+1}\ge q_k$, i.e.\ $2^{-is}\le (2/q_k)^{s}$; summing the geometric series
over that interval gives
\[
  \sum_{i:\,k(i)=k}\frac{q_{k+1}^{\,t}}{2^{is}}
  \;\le\;\frac{2^{s}}{1-2^{-s}}\cdot\frac{q_{k+1}^{\,t}}{q_k^{\,s}} .
\]
Summing over $k\ge0$ yields the right-hand inequality of \eqref{eq:red}.
\end{proof}

\begin{remark}[the cluster at a convergent denominator]\label{rem:cluster}
The lower bound in \eqref{eq:red} uses only the terms $n=q_k$. Each such $n$ in fact
carries a cluster, and in the $\Gst$ normalization the statement is an identity: for
$k\ge1$ put $\delta_k=q_kx-p_k$, so that $\abs{\delta_k}=\normx{q_kx}$ by
\eqref{eq:cf1} and Remark~\ref{rem:k0}. For every integer
$1\le M\le\frac{1}{2\abs{\delta_k}}$ one has $\normx{mq_kx}=m\abs{\delta_k}$ for
$1\le m\le M$, whence
\begin{equation}\label{eq:clusterG}
  \sum_{m=1}^{M}\frac{1}{(mq_k)^{s}\normx{mq_kx}^{t}}
  \;=\;\frac{1}{q_k^{\,s}\normx{q_kx}^{t}}\sum_{m=1}^{M}\frac{1}{m^{s+t}} .
\end{equation}
Since $\abs{\delta_k}<1/q_{k+1}$ by \eqref{eq:cf1}, one may take $M=\lfloor
q_{k+1}/2\rfloor$; the factor $\sum_{m\le M}m^{-s-t}$ therefore approaches
$\zeta(s+t)$ as $q_{k+1}\to\infty$. In the $\Fst$ normalization the corresponding
statement is only asymptotic, and the relevant phase is $\pi\delta_k$, not $\delta_k$:
put
\[
  T_x(n)=\frac{1}{n^{3}\abs{\sin(\pi nx)}^{2}},\qquad
  \theta_k=\pi\bigl(q_kx-p_k\bigr)=\pi\delta_k ,
\]
so that $\sin(\pi mq_kx)=\pm\sin(m\theta_k)$. Then, for $\abs{M\theta_k}\le c<\pi/2$,
\begin{equation}\label{eq:clusterF}
  \frac{T_x(mq_k)}{T_x(q_k)}
  \;=\;m^{-5}\Bigl(\frac{\sin\theta_k/\theta_k}{\sin(m\theta_k)/(m\theta_k)}\Bigr)^{2}
  \;=\;m^{-5}\bigl(1+O_c(m^{2}\theta_k^{2})\bigr)
  \qquad(1\le m\le M).
\end{equation}
For $x=1/\pi$ this reads $\theta_k=q_k-\pi p_k$, and $T_{1/\pi}(n)=(n^{3}\sin^{2}n)^{-1}$
is the Flint Hills summand. For the Flint Hills exponents \eqref{eq:clusterF} predicts a
factor close to $\zeta(5)=1.0369278\ldots$ over such a finite range; see
Table~\ref{tab:cluster}. It says nothing by itself about $\sum_{m\ge1}$, for which
$m\abs{\theta_k}$ eventually exceeds $\pi/2$; the identity \eqref{eq:clusterG} in the
$\normx{\cdot}$ normalization needs no such correction.
\end{remark}

\begin{corollary}\label{cor:crit}
Let $s>0$, $t>1$, and let $x$ be irrational.
\begin{enumerate}
\item[(a)] If $s>t\bigl(\mu(x)-1\bigr)$, equivalently $\mu(x)<1+s/t$, then
$\Fst(x)<\infty$.
\item[(b)] If $s<t\bigl(\mu(x)-1\bigr)$ then $\Fst(x)=\infty$; indeed the terms of
$\Fst(x)$ do not tend to $0$ along the subsequence $n=q_k$.
\item[(c)] If $s\le t$ then $\Fst(x)=\infty$ for \emph{every} irrational $x$.
\end{enumerate}
Part~(a) is Theorem~2.5 of \cite{Meiburg}; the contrapositive of (b) at $(s,t)=(3,2)$ is
Corollary~4 of \cite{Alekseyev}. Parts (b) and (c) use only the left inequality of
\eqref{eq:red} and therefore hold for every $t>0$.
\end{corollary}

\begin{proof}
(a) Write $\beta=\mu(x)-1$ and choose $\varepsilon>0$ with $s>t(\beta+\varepsilon)$.
By \eqref{eq:mu} there is $k_0$ with $q_{k+1}\le q_k^{\beta+\varepsilon}$ for
$k\ge k_0$, whence $q_{k+1}^{t}/q_k^{s}\le q_k^{-\eta}$ with
$\eta=s-t(\beta+\varepsilon)>0$. By \eqref{eq:cf3},
$\sum_{k\ge k_0}q_k^{-\eta}\le\sum_{k}\varphi^{-(k-1)\eta}<\infty$, so
$\Sigma_{s,t}(x)<\infty$ and Proposition~\ref{prop:red} applies.

(b) By \eqref{eq:mu} there are infinitely many $k$ with
$\log q_{k+1}/\log q_k>s/t$, i.e.\ $q_{k+1}^{t}/q_k^{s}>1$; by \eqref{eq:cf1} the
term of $\Gst(x)$ at $n=q_k$ exceeds $q_{k+1}^{t}/q_k^{s}$. Choosing the $k$'s so that
$\log q_{k+1}/\log q_k\ge s/t+\epsilon_0$ for a fixed $\epsilon_0>0$ makes these terms
tend to $\infty$.

(c) Here $q_{k+1}^{t}/q_k^{s}\ge q_k^{\,t-s}\ge1$ for every $k$, so the terms of
$\Sigma_{s,t}(x)$ do not tend to $0$.
\end{proof}

Specializing to $x=1/\pi$ and using Lemma~\ref{lem:inv}:

\begin{corollary}\label{cor:pi}
Let $s>0$ and $S_s=\sum_{n\ge1}\bigl(n^{s}\sin^{2}n\bigr)^{-1}$, so that $S_3=S$ is the
Flint Hills series.
\begin{enumerate}
\item[(a)] $S_s<\infty$ whenever $s>2\mu(\pi)-2$. Since $\mu(\pi)\le7.103205334137\ldots$
\cite{ZeilbergerZudilin}, $S_s<\infty$ unconditionally for every $s>12.20642$; in
particular $\sum_n(n^{13}\sin^2n)^{-1}<\infty$.
\item[(b)] $S_s=\infty$ whenever $s<2\mu(\pi)-2$, and $S_s=\infty$ unconditionally for
every $s\le2$.
\item[(c)] $\mu(\pi)<5/2$ implies $S<\infty$; $\mu(\pi)>5/2$ implies $S=\infty$.
\end{enumerate}
\end{corollary}

Thus the Flint Hills exponent $s=3$ is bracketed by the unconditionally divergent range
$s\le2$ and the unconditionally convergent range $s>12.20642$, and each improvement of
the upper bound for $\mu(\pi)$ lowers the latter threshold; reaching $s=3$ requires
$\mu(\pi)\le5/2$. Nothing in Corollary~\ref{cor:pi} decides the case $s=3$.

\section{The size of the divergence set}\label{sec:dim}

We now let $x$ vary. Recall the Jarn\'ik--Besicovitch theorem
(\cite[\S1.3]{Bugeaud}, \cite[Ch.~10]{Falconer}): for $\tau\ge2$,
\begin{equation}\label{eq:JB}
  \dimH\{x\in\R:\ \mu(x)\ge\tau\}\;=\;\frac{2}{\tau} .
\end{equation}
(The theorem is usually stated for
$K(\tau)=\{x:\abs{x-p/q}<q^{-\tau}\text{ for infinitely many }p/q\}$, whose dimension
is $2/\tau$; since $K(\tau)\subseteq\{\mu\ge\tau\}\subseteq
\bigcap_{0<\varepsilon<\tau-2}K(\tau-\varepsilon)$, the two sets have the same
dimension, which gives \eqref{eq:JB}.)

\begin{theorem}\label{thm:dim}
Let $t>1$, $s>0$ and
\[
  D_{s,t}\;=\;\bigl\{x\in\R\setminus\Q:\ \Fst(x)=+\infty\bigr\}.
\]
Then
\[
  \dimH D_{s,t}\;=\;\min\Bigl\{1,\;\frac{2t}{s+t}\Bigr\} .
\]
Moreover $D_{s,t}=\R\setminus\Q$ when $s\le t$, while for $s>t$ the set $D_{s,t}$ is
dense, of Lebesgue measure zero, and of Hausdorff dimension $2t/(s+t)\in(0,1)$.
In particular, for the Flint Hills exponents,
\[
  \dimH D_{3,2}\;=\;\frac45 .
\]
\end{theorem}

\begin{proof}
If $s\le t$ then $D_{s,t}=\R\setminus\Q$ by Corollary~\ref{cor:crit}(c), so
$\dimH D_{s,t}=1$, in agreement with $2t/(s+t)\ge1$.

Let $s>t$ and put $\tau_c=1+s/t>2$. By Corollary~\ref{cor:crit}(a),(b),
\begin{equation}\label{eq:sandwich}
  \{x\in\R\setminus\Q:\ \mu(x)>\tau_c\}\ \subseteq\ D_{s,t}
  \ \subseteq\ \{x\in\R\setminus\Q:\ \mu(x)\ge\tau_c\}.
\end{equation}
The upper bound in \eqref{eq:JB} gives $\dimH D_{s,t}\le 2/\tau_c$. For the lower
bound, write $\{\mu>\tau_c\}=\bigcup_{m\ge1}\{\mu\ge\tau_c+1/m\}$; countable unions do
not increase Hausdorff dimension beyond the supremum, so by \eqref{eq:JB}
\[
  \dimH\{\mu>\tau_c\}=\sup_{m\ge1}\frac{2}{\tau_c+1/m}=\frac{2}{\tau_c},
\]
and hence $\dimH D_{s,t}\ge 2/\tau_c$. Therefore
$\dimH D_{s,t}=2/\tau_c=2t/(s+t)$.

Finally, $\mu(x)=2<\tau_c$ for almost every $x$ by Khinchin's theorem, so $D_{s,t}$ is
Lebesgue-null; and $D_{s,t}$ contains every Liouville number, hence is dense. For
$s=3$, $t=2$ we get $\tau_c=5/2$ and $\dimH D_{3,2}=4/5$.
\end{proof}

\begin{remark}
The two features of $D_{3,2}$ pull in opposite directions. It is null, which is the
statistical ground for expecting $S<\infty$: a number drawn at random converges. But it
has dimension $4/5$ and is dense, so its size alone carries no information about any
particular point of $\R$. In particular Theorem~\ref{thm:dim} does \emph{not} decide the
convergence of $S$, and no sharpening of the \emph{size} of $D_{3,2}$ could: both
answers for $1/\pi$ are consistent with every statement about $D_{3,2}$ as a set.
Deciding membership of a prescribed number is a different kind of question, requiring
effective Diophantine information about that number.
\end{remark}

\begin{remark}
As $s$ increases with $t=2$ fixed, the divergence set shrinks:
\[
  \dimH D_{s,2}=\frac{4}{s+2}\;:\qquad
  s=2\mapsto1,\quad 3\mapsto\tfrac45,\quad 4\mapsto\tfrac23,\quad 6\mapsto\tfrac12,
  \quad 12.20642\mapsto 0.28156\ldots
\]
The last value is the exponent at which Corollary~\ref{cor:pi}(a) becomes
unconditional for $\pi$. Even there the exceptional set has positive dimension; what
makes the statement unconditional is the hypergeometric bound on $\mu(\pi)$, not the
smallness of $D_{s,2}$.
\end{remark}

\section{The critical exponent}\label{sec:sharp}

Corollary~\ref{cor:crit} leaves open the borderline $s=t(\mu(x)-1)$, i.e.\
$\mu(x)=\tau_c$. Both behaviours occur there, and a single family exhibits them.

\begin{proposition}\label{prop:sharp}
Let $t>1$ and $\tau>2$, and put $\beta=\tau-1>1$ and $s=t\beta$. For $r\ge0$ define
$x_r=[0;a_1,a_2,\dots]$ recursively by
\begin{equation}\label{eq:constr}
  a_1=2,\qquad
  a_{k+1}=\Bigl\lceil \frac{q_k^{\,\beta-1}}{(k+1)^{r/t}}\Bigr\rceil \quad(k\ge1).
\end{equation}
Then $x_r$ is a well-defined irrational number with
\[
  \frac{q_{k+1}^{\,t}}{q_k^{\,s}}\;=\;\frac{1}{(k+1)^{r}}\bigl(1+o(1)\bigr),
  \qquad
  \frac{\log q_{k+1}}{\log q_k}\;\longrightarrow\;\beta ,
\]
so that $\mu(x_r)=\tau$ for every $r\ge0$, while
\[
  \Fst(x_r)<\infty \iff r>1 .
\]
In particular $x_2$ and $x_1$ both have irrationality exponent exactly $\tau$, and
$\Fst(x_2)<\infty$ while $\Fst(x_1)=\infty$. Moreover the summands of $\Fst(x_r)$
tend to $0$ for every $r>0$, whereas for $r=0$ they do not: for $q_k\le n<q_{k+1}$ one
has $\normx{nx_r}\ge\normx{q_kx_r}>1/(q_{k+1}+q_k)$ by
\eqref{eq:cf1}--\eqref{eq:cf2}, hence
\begin{equation}\label{eq:termsize}
  \frac{1}{n^{s}\abs{\sin(\pi nx_r)}^{t}}
  \;\le\;\frac{(q_{k+1}+q_k)^{t}}{2^{t}\,q_k^{\,s}}
  \;=\;O\bigl((k+1)^{-r}\bigr),
\end{equation}
which tends to $0$ when $r>0$. For $r=0$, by contrast, \eqref{eq:FG} and \eqref{eq:cf1}
give at $n=q_k$
\[
  \frac{1}{q_k^{\,s}\abs{\sin(\pi q_kx_0)}^{t}}
  \;\ge\;\pi^{-t}\,\frac{1}{q_k^{\,s}\normx{q_kx_0}^{t}}
  \;>\;\pi^{-t}\,\frac{q_{k+1}^{t}}{q_k^{\,s}}\;\longrightarrow\;\pi^{-t},
\]
and in fact the terms at $n=q_k$ converge to $\pi^{-t}$, since $a_{k+1}\to\infty$ forces
$q_k/q_{k+1}\to0$ and hence $q_{k+1}\normx{q_kx_0}\to1$ by \eqref{eq:cf1}. Along
$m_k=q_k+1$ one has $\normx{m_kx_0}\to\normx{x_0}>0$, so those terms tend to $0$.
The sequence of summands of $\mathcal F_{s,t}(x_0)$ therefore has the two subsequential
limits $\pi^{-t}$ and $0$ and does not converge.
\end{proposition}

\begin{proof}
Every $a_{k+1}$ in \eqref{eq:constr} is a positive integer, and any sequence of positive
integers is the partial quotient sequence of a unique irrational in $(0,1)$; the
denominators are determined by $q_{k+1}=a_{k+1}q_k+q_{k-1}$, so $x_r$ is well defined.
Put
\[
  u_k=\frac{q_k^{\,\beta-1}}{(k+1)^{r/t}} .
\]
By \eqref{eq:cf3}, $\log u_k\ge(\beta-1)(k-1)\log\varphi-(r/t)\log(k+1)\to\infty$, so
$u_k\to\infty$. From $u_k\le a_{k+1}\le u_k+1$ and $q_{k-1}<q_k$,
\[
  u_kq_k\ \le\ q_{k+1}\ \le\ (u_k+2)q_k\ =\ u_kq_k\Bigl(1+\frac{2}{u_k}\Bigr),
\]
i.e.\ $q_{k+1}=u_kq_k\bigl(1+O(u_k^{-1})\bigr)$. Since $\beta t=s$,
\[
  \bigl(u_kq_k\bigr)^{t}q_k^{-s}
  \;=\;u_k^{t}\,q_k^{\,t-s}
  \;=\;q_k^{(\beta-1)t}(k+1)^{-r}q_k^{\,t-s}
  \;=\;(k+1)^{-r},
\]
so $q_{k+1}^{t}/q_k^{s}=(k+1)^{-r}\bigl(1+O(u_k^{-1})\bigr)^{t}$, which is the first
assertion. Hence $\Sigma_{s,t}(x_r)<\infty$ if and only if $\sum_k(k+1)^{-r}<\infty$,
i.e.\ if and only if $r>1$, and Proposition~\ref{prop:red} converts this into the
statement about $\Fst(x_r)$.

For the exponent, taking logarithms in $q_{k+1}=u_kq_k(1+O(u_k^{-1}))$ gives
\[
  \frac{\log q_{k+1}}{\log q_k}
  \;=\;\beta-\frac{r}{t}\cdot\frac{\log(k+1)}{\log q_k}+o(1)
  \;\longrightarrow\;\beta ,
\]
because $\log q_k\ge(k-1)\log\varphi$ grows at least linearly while $\log(k+1)$ does
not. By \eqref{eq:mu}, $\mu(x_r)=1+\beta=\tau$.
\end{proof}

\begin{remark}
The hypothesis $\tau>2$ is necessary for the coexistence: at $\tau=2$ one has $s=t$ and
every irrational diverges by Corollary~\ref{cor:crit}(c).
\end{remark}

\begin{remark}\label{rem:meiburg}
Write $u=s$ and $v=t$ for the notation of \cite{Meiburg}. For $u>v>1$ the member $x_0$ recovers
\cite[Thm.~3.1]{Meiburg} in full, since for $r=0$ both the series and its summands fail
to converge. The members $x_r$ with $0<r\le1$ give the strictly weaker
series-divergence conclusion, but they are the more informative ones: by
\eqref{eq:termsize} their summands \emph{do} tend to $0$, so these are numbers of
irrationality exponent exactly $\tau$ at which $\Fst$ diverges without any single term
misbehaving. The convergent members $x_r$, $r>1$, settle \cite[Conj.~3.1]{Meiburg} in
the range $v>1$.

The conjecture concerns not $\sin$ alone but any sine-like function $P$ of period
$\alpha$, and the construction covers that formulation as well. Take $x=x_r$ and
$\alpha=1/x$. The sine-like condition provides constants $c_P,C_P>0$ with
$c_P\alpha\normx{nx}\le\abs{P(n)}\le C_P\alpha\normx{nx}$ for all $n$, whence
\[
  (C_P\alpha)^{-v}\,\mathcal G_{u,v}(x)
  \;\le\;\sum_{n\ge1}\frac{1}{n^{u}\abs{P(n)}^{v}}
  \;\le\;(c_P\alpha)^{-v}\,\mathcal G_{u,v}(x) .
\]
So a single $x_r$ with $r>1$ gives convergence simultaneously for every sine-like $P$ of
that period, and $\mu(\alpha)=\mu(x_r)=1+u/v$ by Lemma~\ref{lem:inv}, which is what
\cite[Conj.~3.1]{Meiburg} asks for.

The range $0<v\le1$ is not covered here, and there the conjecture as stated is false:
with the normalization $\abs{\sin}\le1$, $\mathcal F_{u,v}(x)\ge\sum_{n\ge1}n^{-u}=\infty$
whenever $0<u\le1$, so no convergent example exists for such $u$, while the hypothesis
$1+u/v>2$ is satisfied throughout the non-empty region $v<u\le1$ --- for instance
$(u,v)=(3/4,1/2)$. The necessary additional condition is $u>1$. For $v\le1$ and $u>1$
the question remains open: the upper bound of Proposition~\ref{prop:red} uses
$\zeta(t)<\infty$, and for $t\le1$ the dyadic blocks contribute at the same order as the
spikes, so the criterion itself must be reformulated.
\end{remark}

\begin{corollary}\label{cor:noiff}
Let $t>1$ and $s>t$, and put $\tau_c=1+s/t$. Then
\[
  \mu(x)<\tau_c\ \Longrightarrow\ \Fst(x)<\infty\ \Longrightarrow\ \mu(x)\le\tau_c ,
\]
and neither implication can be reversed: $x_2$ of Proposition~\ref{prop:sharp} has
$\mu=\tau_c$ and $\Fst<\infty$, so the first fails to be an equivalence, and $x_1$ has
$\mu=\tau_c$ and $\Fst=\infty$, so the second fails as well. Consequently no condition
on $\mu(x)$ alone --- with strict or non-strict inequality --- is equivalent to the
convergence of $\Fst(x)$. A necessary and sufficient condition is available, but it is
not a function of $\mu$: by Proposition~\ref{prop:red},
$\Fst(x)<\infty\iff\sum_kq_{k+1}^{t}/q_k^{s}<\infty$.
\end{corollary}

At $(s,t)=(3,2)$ this says that ``$S$ converges $\iff\mu(\pi)\le5/2$'' cannot be an
instance of a general theorem about numbers of irrationality exponent $5/2$, and that
replacing $\le$ by $<$ does not repair matters. For the individual number $1/\pi$ such
an equivalence is of course not excluded. The preprint \cite{LopezZapata} asserts one;
in the version we consulted, the convergence criterion of its Lemma~3.3 carries a
strict inequality, and the subsequent derivation of the equivalence does not treat the
case of equality, which is exactly the case Proposition~\ref{prop:sharp} shows to be
undetermined by $\mu$.

\section{The critical fibre}\label{sec:fibre}

Corollary~\ref{cor:noiff} says that $\mu$ does not decide convergence on the critical
fibre $\{\mu=\tau_c\}$. One may ask whether the two parts of that fibre differ in size.
Their Hausdorff dimensions agree; a finer gauge does tell them apart. Throughout this section $t>1$, $\tau>2$, $\beta=\tau-1$ and $s=t\beta$, and
\[
  C_\tau=\{x\in\R\setminus\Q:\ \mu(x)=\tau,\ \Fst(x)<\infty\},\qquad
  E_\tau=\{x\in\R\setminus\Q:\ \mu(x)=\tau,\ \Fst(x)=\infty\},
\]
so that $\{\mu=\tau\}=C_\tau\sqcup E_\tau$, and both are non-empty by
Proposition~\ref{prop:sharp}.

The following elementary identity converts the criterion of
Proposition~\ref{prop:red} into a statement about the quality of the rational
approximations themselves, which is the form in which the metric theory can be applied.

\begin{lemma}\label{lem:bridge}
For irrational $x$ and $k\ge1$ put $\rho_k(x)=q_k^{\,\tau}\abs{x-p_k/q_k}$. Then
\[
  \bigl(2\rho_k\bigr)^{-t}\;<\;\frac{q_{k+1}^{\,t}}{q_k^{\,s}}\;<\;\rho_k^{-t},
\]
and consequently $\Sigma_{s,t}(x)<\infty$ if and only if
$\sum_k\rho_k(x)^{-t}<\infty$.
\end{lemma}

\begin{proof}
By \eqref{eq:cf1}, $\frac{1}{q_k(q_{k+1}+q_k)}<\abs{x-p_k/q_k}<\frac{1}{q_kq_{k+1}}$.
Multiplying by $q_k^{\tau}=q_k^{1+\beta}$ and using $q_{k+1}+q_k<2q_{k+1}$ gives
\[
  \frac{q_k^{\beta}}{2q_{k+1}}\;<\;\rho_k\;<\;\frac{q_k^{\beta}}{q_{k+1}} ,
  \qquad\text{hence, inverting,}\qquad
  \frac{1}{2}\,\rho_k^{-1}\;<\;\frac{q_{k+1}}{q_k^{\beta}}\;<\;\rho_k^{-1} .
\]
Raising to the power $t$ and recalling
$q_{k+1}^{t}/q_k^{s}=(q_{k+1}/q_k^{\beta})^{t}$ gives the display; the last assertion
follows since the two sides differ by the factor $2^{t}$.
\end{proof}

The size of $C_\tau$ and $E_\tau$ is determined in Section~\ref{sec:gauge}: they have
the same Hausdorff dimension $2/\tau$ (Corollary~\ref{thm:fibre}), but they are told
apart by a finer gauge (Theorem~\ref{thm:gaugefibre}). Nothing in that section uses the
theory of exact approximation order; we record the alternative route, which was how we
first obtained Corollary~\ref{thm:fibre}, because it is shorter. For
$\psi:\N\to(0,\infty)$ set
\[
  W(\psi)=\Bigl\{x\in\R:\ \bigl|x-\tfrac pq\bigr|<\psi(q)
  \text{ for infinitely many }(p,q)\in\mathbb Z\times\N\Bigr\},
\]
\[
  E(\psi)=W(\psi)\setminus\bigcup_{0<c<1}W(c\psi),
\]
so that $x\in E(\psi)$ if and only if $\abs{x-p/q}<\psi(q)$ has infinitely many
solutions while, for each $c\in(0,1)$, $\abs{x-p/q}\ge c\,\psi(q)$ for all sufficiently
large $q$ and all $p$.

\begin{theorem}[Bugeaud; stated as in {\cite[\S1.3]{FraserWheeler}}]\label{thm:bugeaud}
Suppose that $q\mapsto q^{2}\psi(q)$ is non-increasing and that
$\sum_{q\ge1}q\,\psi(q)<\infty$, and put
$\lambda=\liminf_{q\to\infty}\log(1/\psi(q))/\log q$. Then
$\dimH E(\psi)=\dimH W(\psi)=2/\lambda$.
\end{theorem}

\begin{remark}\label{rem:viaexact}
Theorem~\ref{thm:bugeaud} yields Corollary~\ref{thm:fibre} directly, as follows; it is
not used anywhere else in this note. Both $C_\tau$ and $E_\tau$ lie in $\{\mu\ge\tau\}$,
so their dimensions are at most $2/\tau$ by \eqref{eq:JB}. For the lower bounds take
\[
  \psi_E(q)=q^{-\tau},
  \qquad
  \psi_C(q)=q^{-\tau}(\log q)^{\gamma}\ \ (\gamma>1/t),
\]
each modified on a finite range so that $q^{2}\psi$ is non-increasing; both satisfy the
hypotheses with $\lambda=\tau$, so $\dimH E(\psi_E)=\dimH E(\psi_C)=2/\tau$. If
$x\in E(\psi_E)$ then $\abs{x-p/q}<q^{-\tau}<\frac12q^{-2}$ for infinitely many reduced
$p/q$, which are then convergents by Legendre's theorem, so $\rho_k(x)\le1$ infinitely
often and $\Fst(x)=\infty$ by Lemma~\ref{lem:bridge}; the exactness gives $\mu(x)=\tau$,
whence $E(\psi_E)\subseteq E_\tau$. If $x\in E(\psi_C)$ then, taking $c=\frac12$,
$\rho_k(x)\ge\frac12(\log q_k)^{\gamma}$ for all large $k$, so
$\sum_k\rho_k(x)^{-t}<\infty$ by \eqref{eq:cf3} and $\gamma t>1$; again $\mu(x)=\tau$,
whence $E(\psi_C)\subseteq C_\tau$.
\end{remark}

\section{A zero--infinity law and the gauge that separates}\label{sec:gauge}

Ordinary Hausdorff dimension does not see the difference between $C_\tau$ and $E_\tau$
--- both have dimension $2/\tau$, as Corollary~\ref{cor:nodim} below records. A gauge
does, and the threshold is exactly $\kappa=-1$.
Throughout, $t>1$, $s>t$, $\tau=1+s/t>2$, and for $\kappa\in\R$ we use the dimension
function
\begin{equation}\label{eq:gauge}
  h_\kappa(r)\;=\;r^{2/\tau}\Bigl(\log\frac1r\Bigr)^{\kappa}
  \qquad (0<r<r_0),
\end{equation}
extended arbitrarily and monotonically to $[r_0,\infty)$. For every $\kappa\in\R$, after
shrinking $r_0$, $h_\kappa$ is increasing with $h_\kappa(0^+)=0$ and
$h_\kappa(2r)\asymp h_\kappa(r)$, and $r^{-1}h_\kappa(r)$ is decreasing: writing
$a=2/\tau<1$ and $L=\log(1/r)$ one has
$\frac{d}{dr}\bigl(r^{a}L^{\kappa}\bigr)=r^{a-1}L^{\kappa-1}(aL-\kappa)>0$ once
$L>\kappa/a$, while $r^{-1}h_\kappa(r)=r^{a-1}L^{\kappa}$ is decreasing because
$a-1<0$ dominates. Also $h_\kappa(r)/r\to\infty$, so
$\mathcal H^{h_\kappa}([0,1])=\infty$. Finally, if $\kappa\ge\kappa_0$ then
$h_\kappa\ge h_{\kappa_0}$ on $(0,e^{-1})$, whence
\begin{equation}\label{eq:monogauge}
  \mathcal H^{h_\kappa}(A)\ \ge\ \mathcal H^{h_{\kappa_0}}(A)
  \qquad\text{for every }A\subseteq\R,\ \kappa\ge\kappa_0 .
\end{equation}

We use the two classical halves of the Hausdorff-measure theory of the limsup sets
\[
  W^{*}(\psi)=\Bigl\{x\in\R:\ \bigl|x-\tfrac pq\bigr|<\psi(q)
  \text{ for infinitely many coprime }(p,q),\ q\ge1\Bigr\}.
\]

\begin{lemma}[Hausdorff--Cantelli]\label{lem:HC}
Let $h$ be a dimension function with $h(2r)\le Ch(r)$ for small $r$, and let
$\psi(q)\to0$. If $\sum_{q\ge1}q\,h(\psi(q))<\infty$ then
$\mathcal H^{h}\bigl(W^{*}(\psi)\bigr)=0$.
\end{lemma}

\begin{proof}
Everything is $1$-periodic, so it suffices to treat $W^*(\psi)\cap[0,1)$, which for every
$Q$ is covered by those intervals $B(p/q,\psi(q))$ with $q\ge Q$ that meet $[0,1)$. Once
$q$ is large enough that $\psi(q)\le1$ there are at most $3q+3$ such $p$ for each $q$,
and the intervals have diameters at most $\delta_Q=2\sup_{q\ge Q}\psi(q)\to0$, so they
form an admissible cover at every scale $\delta_Q$, and
\[
  \mathcal H^{h}_{\delta_Q}\bigl(W^*(\psi)\cap[0,1)\bigr)
  \;\le\;C\sum_{q\ge Q}(3q+3)\,h(\psi(q))\;\xrightarrow[Q\to\infty]{}\;0 .
\]
Letting $Q\to\infty$ gives $\mathcal H^{h}(W^*(\psi)\cap[0,1))=0$.
\end{proof}

\begin{theorem}[Jarn\'ik; see \cite{BDV,BeresnevichVelani}]\label{thm:jarnik}
Let $\psi$ be such that $q\mapsto q\,\psi(q)$ is non-increasing, and let $h$ be a
dimension function with $r^{-1}h(r)$ monotonic. If $\sum_{q\ge1}q\,h(\psi(q))=\infty$
then
\[
  \mathcal H^{h}\bigl(W^{*}(\psi)\cap[0,1]\bigr)=\mathcal H^{h}\bigl([0,1]\bigr).
\]
In particular, if moreover $h(r)/r\to\infty$ as $r\to0$, then
$\mathcal H^{h}(W^{*}(\psi))=\infty$.
\end{theorem}

The additional hypothesis $h(r)/r\to\infty$ is needed for the infinity conclusion, not
for the measure equality: $h(r)=r^{2}$ and $\psi(q)=q^{-1}$ satisfy the displayed
hypotheses with $\sum_qq\,h(\psi(q))=\sum_qq^{-1}=\infty$, and the equality holds with
both sides equal to $\mathcal H^{2}([0,1])=0$. The intersection with $[0,1]$ cannot be
dropped: $h(r)=r$ and $\psi(q)=q^{-2}$ satisfy them, and
$W^{*}(\psi)$ contains every irrational by Dirichlet's theorem, so the left-hand side is
$\infty$ while $\mathcal H^{1}([0,1])=1$. For the gauges \eqref{eq:gauge} one has
$h_\kappa(r)/r\to\infty$ because $2/\tau<1$, so the ``in particular'' applies.

The next lemma is the only new ingredient. It converts the divergence of
$\sum_k\rho_k^{-t}$ into a single limsup condition; the point is to discard the indices
that cannot contribute, after which the denominators grow doubly exponentially and the
index $k$ can be replaced by $\log\log q_k$ rather than by $\log q_k$.

\begin{lemma}\label{lem:Dcover}
Let $t>1$, $s>t$, $\tau=1+s/t>2$, and set
\[
  \eta=\frac{\tau-2}{2}>0,\qquad b=1+\frac{\eta}{2}>1,\qquad
  \Lambda_\tau=1+\frac{1}{\log b},\qquad Q_{*}=\max\bigl\{e^{e},\,2^{2/\eta}\bigr\}.
\]
Then for every $\varepsilon>0$,
\[
  D_{s,t}\ \subseteq\ W^{*}(\widehat\psi_\varepsilon),
  \qquad
  \widehat\psi_\varepsilon(q)=\Lambda_\tau^{\,a}\,q^{-\tau}\bigl(\log\log q\bigr)^{a},
  \quad a=\frac{1+\varepsilon}{t},
\]
where $\widehat\psi_\varepsilon$ is defined by this formula for $q\ge Q_*$ and by any
positive values for the finitely many $q<Q_*$; Lemma~\ref{lem:HC}, which is where it is
used, requires no monotonicity.
\end{lemma}

\begin{proof}
Fix $x\in D_{s,t}$. By Proposition~\ref{prop:red} we have $\Sigma_{s,t}(x)=\infty$, and
by Lemma~\ref{lem:bridge}, $q_{k+1}^{t}/q_k^{s}<\rho_k^{-t}$, so
\begin{equation}\label{eq:divrho}
  \sum_{k\ge1}\rho_k^{-t}=\infty,\qquad \rho_k=\rho_k(x)=q_k^{\tau}\Bigl|x-\frac{p_k}{q_k}\Bigr| .
\end{equation}

\emph{Step 1: discard the harmless indices.} Put $I=\{k\ge1:\rho_k<q_k^{\eta}\}$. For
$k\notin I$ we have $\rho_k^{-t}\le q_k^{-\eta t}$, and $\sum_kq_k^{-\eta t}<\infty$ by
\eqref{eq:cf3}. Hence, by \eqref{eq:divrho},
\[
  \sum_{k\in I}\rho_k^{-t}=\infty ;
\]
in particular $I$ is infinite.

\emph{Step 2: on $I$ the denominators jump polynomially.} The left inequality in the
proof of Lemma~\ref{lem:bridge} reads $\rho_k>q_k^{\tau-1}/(2q_{k+1})$, i.e.
$q_{k+1}>q_k^{\tau-1}/(2\rho_k)$. For $k\in I$ this gives
\[
  q_{k+1}\;>\;\frac{q_k^{\tau-1}}{2q_k^{\eta}}\;=\;\frac12\,q_k^{1+\eta},
\]
because $\tau-1-\eta=1+\eta$. If in addition $q_k\ge Q_*$ then $q_k^{\eta/2}\ge2$, so
$\frac12q_k^{1+\eta}\ge q_k^{1+\eta/2}=q_k^{b}$ and therefore
\begin{equation}\label{eq:jump}
  q_{k+1}\;>\;q_k^{\,b}\qquad\bigl(k\in I,\ q_k\ge Q_*\bigr).
\end{equation}
Enumerate $\{k\in I:\,q_k\ge Q_*\}$ as $k_1<k_2<\cdots$ and put $Q_j=q_{k_j}$,
$R_j=\rho_{k_j}$. Since $k_{j+1}\ge k_j+1$ and $(q_k)$ increases, \eqref{eq:jump} gives
$Q_{j+1}\ge q_{k_j+1}>Q_j^{\,b}$, hence $\log Q_j\ge b^{\,j-1}\log Q_1$ and
\[
  j\;\le\;1+\frac{\log\log Q_j-\log\log Q_1}{\log b}\;\le\;\Lambda_\tau\log\log Q_j ,
\]
where the last step uses $\log\log Q_1\ge1$, valid because $Q_1\ge Q_*\ge e^{e}$. Note
that $\Lambda_\tau$ and $Q_*$ depend only on $\tau$, not on $x$.

\emph{Step 3: transfer the divergence to a limsup condition.} Only finitely many
$k\in I$ have $q_k<Q_*$, so $\sum_jR_j^{-t}=\infty$ by Step~1. As
$\sum_jj^{-(1+\varepsilon)}<\infty$, the inequality $R_j^{-t}\le j^{-(1+\varepsilon)}$
cannot hold for all large $j$; hence $R_j<j^{a}$ for infinitely many $j$, and by Step~2,
\[
  R_j\;<\;\bigl(\Lambda_\tau\log\log Q_j\bigr)^{a}
  \qquad\text{for infinitely many }j .
\]
For those $j$,
$\bigl|x-p_{k_j}/Q_j\bigr|=R_jQ_j^{-\tau}<\widehat\psi_\varepsilon(Q_j)$, and the pairs
$(p_{k_j},Q_j)$ are coprime and distinct. Thus $x\in W^{*}(\widehat\psi_\varepsilon)$.
\end{proof}

\begin{theorem}\label{thm:gaugeD}
Let $t>1$, $s>t$ and $\tau=1+s/t$. Then, with $h_\kappa$ as in \eqref{eq:gauge},
\[
  \mathcal H^{h_\kappa}\bigl(D_{s,t}\bigr)=
  \begin{cases}
    0, & \kappa<-1,\\[2pt]
    \infty, & \kappa\ge-1 .
  \end{cases}
\]
\end{theorem}

\begin{proof}
\emph{Zero for $\kappa<-1$.} Fix any $\varepsilon>0$ and let
$\widehat\psi_\varepsilon$ be as in Lemma~\ref{lem:Dcover}, with $a=(1+\varepsilon)/t$.
Since $\log\bigl(1/\widehat\psi_\varepsilon(q)\bigr)=\tau\log q+O(\log\log\log q)$,
\[
  q\,h_\kappa\bigl(\widehat\psi_\varepsilon(q)\bigr)
  \;\asymp\; q\cdot q^{-2}\bigl(\log\log q\bigr)^{2a/\tau}\bigl(\log q\bigr)^{\kappa}
  \;=\;\frac{(\log q)^{\kappa}\bigl(\log\log q\bigr)^{2a/\tau}}{q} .
\]
Substituting $u=\log q$, the associated integral is
$\int^{\infty}u^{\kappa}(\log u)^{2a/\tau}\,du$, which converges precisely when
$\kappa<-1$, for every fixed exponent $2a/\tau$. So
$\sum_qq\,h_\kappa(\widehat\psi_\varepsilon(q))<\infty$, and Lemma~\ref{lem:HC} gives
$\mathcal H^{h_\kappa}\bigl(W^{*}(\widehat\psi_\varepsilon)\bigr)=0$;
Lemma~\ref{lem:Dcover} concludes.

\emph{Infinity for $\kappa\ge-1$.} We claim $W^{*}(q^{-\tau})\subseteq D_{s,t}\cup\Q$.
Indeed, if $x$ is irrational and $\abs{x-p/q}<q^{-\tau}$ for infinitely many coprime
pairs, then $q^{-\tau}<\frac12q^{-2}$ for large $q$, so each such $p/q$ is a convergent
by Legendre's theorem; hence $\rho_k(x)<1$ for infinitely many $k$ and, by
Lemma~\ref{lem:bridge}, infinitely many terms of $\Sigma_{s,t}(x)$ exceed $2^{-t}$.
Thus $\Sigma_{s,t}(x)=\infty$ and, by the left inequality of \eqref{eq:red} together
with \eqref{eq:FG}, $\Fst(x)=\infty$. Now $q\cdot q^{-\tau}$ is non-increasing and
\[
  \sum_qq\,h_\kappa\bigl(q^{-\tau}\bigr)\;\asymp\;\sum_q\frac{(\log q)^{\kappa}}{q}\;=\;\infty
  \qquad(\kappa\ge-1),
\]
so $\mathcal H^{h_\kappa}(W^{*}(q^{-\tau}))=\infty$ by Theorem~\ref{thm:jarnik}, the
extra hypothesis $h_\kappa(r)/r\to\infty$ being satisfied. Since $\Q$ is countable and
$h_\kappa(0^{+})=0$, $\mathcal H^{h_\kappa}(D_{s,t})=\infty$.
\end{proof}

\begin{theorem}\label{thm:gaugefibre}
Let $t>1$, $\tau>2$ and $s=t(\tau-1)$. Then
\[
  \mathcal H^{h_\kappa}\bigl(C_\tau\bigr)=\infty\quad\text{for every }\kappa\in\R,
  \qquad
  \mathcal H^{h_\kappa}\bigl(E_\tau\bigr)=
  \begin{cases}
    0, & \kappa<-1,\\[2pt]
    \infty, & \kappa\ge-1 .
  \end{cases}
\]
In particular every $h_\kappa$ with $\kappa<-1$ separates the two parts of the critical
fibre, in the strongest possible sense, and $\kappa=-1$ is the exact threshold at which
the separation disappears.
\end{theorem}

\begin{proof}
\emph{The set $C_\tau$.} Fix any $\kappa_0<-1$ and put $\gamma=\tau(-1-\kappa_0)/2>0$
and $\psi_\gamma(q)=q^{-\tau}(\log q)^{\gamma}$. Then $q\psi_\gamma(q)$ is
non-increasing for large $q$; since altering $\psi$ at finitely many $q$ does not change
$W^{*}(\psi)$, we may assume this throughout. As
$\log(1/\psi_\gamma(q))=\tau\log q+O(\log\log q)$,
\[
  q\,h_{\kappa_0}\bigl(\psi_\gamma(q)\bigr)
  \;\asymp\;\frac{(\log q)^{2\gamma/\tau+\kappa_0}}{q}
  \;=\;\frac{1}{q\log q},
\]
because $2\gamma/\tau+\kappa_0=(-1-\kappa_0)+\kappa_0=-1$; and
$\sum_q1/(q\log q)=\infty$. Theorem~\ref{thm:jarnik} therefore gives
$\mathcal H^{h_{\kappa_0}}(W^{*}(\psi_\gamma))=\infty$.

Next, $W^{*}(\psi_\gamma)\subseteq\{\mu\ge\tau\}$: for $\tau'<\tau$ one has
$\psi_\gamma(q)\le q^{-\tau'}$ for all large $q$, so any $x\in W^{*}(\psi_\gamma)$
satisfies $\mu(x)\ge\tau'$, and $\tau'\uparrow\tau$ gives $\mu(x)\ge\tau$. Hence if
$x\in W^{*}(\psi_\gamma)$ is irrational and $x\notin D_{s,t}$, then $\Fst(x)<\infty$,
which by Corollary~\ref{cor:crit}(b) forces $\mu(x)\le\tau$; so $\mu(x)=\tau$ and
$x\in C_\tau$. Therefore
$W^{*}(\psi_\gamma)\setminus(D_{s,t}\cup\Q)\subseteq C_\tau$, and since
$\mathcal H^{h_{\kappa_0}}(D_{s,t})=0$ by Theorem~\ref{thm:gaugeD},
\[
  \mathcal H^{h_{\kappa_0}}\bigl(C_\tau\bigr)
  \;\ge\;\mathcal H^{h_{\kappa_0}}\bigl(W^{*}(\psi_\gamma)\bigr)\;=\;\infty .
\]
For an arbitrary $\kappa\in\R$ choose $\kappa_0<\min\{\kappa,-1\}$; then
\eqref{eq:monogauge} gives $\mathcal H^{h_\kappa}(C_\tau)=\infty$.

\emph{The set $E_\tau$.} For $\kappa<-1$ we have $E_\tau\subseteq D_{s,t}$ and
Theorem~\ref{thm:gaugeD} applies. For $\kappa\ge-1$, note first that
$\{\mu>\tau\}=\bigcup_{m\ge1}\{\mu\ge\tau+1/m\}$, and by \eqref{eq:JB} each of these
sets has Hausdorff dimension $2/(\tau+1/m)<2/\tau$; choosing $d'$ with
$2/(\tau+1/m)<d'<2/\tau$ we have $h_\kappa(r)\le r^{d'}$ for small $r$, so
$\mathcal H^{h_\kappa}(\{\mu\ge\tau+1/m\})\le\mathcal H^{d'}(\{\mu\ge\tau+1/m\})=0$ and
hence $\mathcal H^{h_\kappa}(\{\mu>\tau\})=0$. By the proof of
Theorem~\ref{thm:gaugeD}, every irrational $x\in W^{*}(q^{-\tau})$ lies in $D_{s,t}$ and
satisfies $\mu(x)\ge\tau$; so
\[
  W^{*}(q^{-\tau})\setminus\bigl(\{\mu>\tau\}\cup\Q\bigr)\ \subseteq\ E_\tau ,
\]
and $\mathcal H^{h_\kappa}(E_\tau)\ge\mathcal H^{h_\kappa}(W^{*}(q^{-\tau}))=\infty$.
\end{proof}

The three sets are therefore distinguished as follows.
\[
  \begin{array}{c|ccc}
    & \mathcal H^{h_\kappa}(D_{s,t}) & \mathcal H^{h_\kappa}(E_\tau)
    & \mathcal H^{h_\kappa}(C_\tau)\\[2pt]
    \hline
    \kappa<-1 & 0 & 0 & \infty\\[2pt]
    \kappa\ge-1 & \infty & \infty & \infty
  \end{array}
\]

Taking $\kappa=0$ gives $\mathcal H^{2/\tau}(C_\tau)=\mathcal H^{2/\tau}(E_\tau)=\infty$,
which pins down the dimensions as well.

\begin{corollary}\label{thm:fibre}
For every $t>1$ and $\tau>2$, with $s=t(\tau-1)$,
\[
  \dimH C_\tau\;=\;\dimH E_\tau\;=\;\frac{2}{\tau}\;=\;\dimH\{x:\mu(x)=\tau\} .
\]
\end{corollary}

\begin{proof}
By Theorem~\ref{thm:gaugefibre} with $\kappa=0$, both sets have infinite
$\mathcal H^{2/\tau}$-measure, so both have Hausdorff dimension at least $2/\tau$. Both
lie in $\{\mu\ge\tau\}$, so \eqref{eq:JB} bounds their dimensions by $2/\tau$. Finally
$\{\mu=\tau\}=C_\tau\cup E_\tau$.
\end{proof}

\begin{corollary}\label{cor:nodim}
Ordinary Hausdorff dimension does not distinguish convergence from divergence on the
critical fibre: $C_\tau$, $E_\tau$ and $\{\mu=\tau\}$ all have dimension $2/\tau$. So
neither the irrationality exponent of $x$ (Corollary~\ref{cor:noiff}) nor the dimension
of the part of the fibre it lies in carries the information deciding the convergence of
$\Fst(x)$, whereas every $h_\kappa$ with $\kappa<-1$ does.
\end{corollary}

\begin{remark}
The divergent half of Theorem~\ref{thm:gaugeD} uses only the left inequality of
\eqref{eq:red}, so $\mathcal H^{h_\kappa}(D_{s,t})=\infty$ for $\kappa\ge-1$ holds for
every $t>0$. Everything else in this section inherits the restriction $t>1$ from the
upper bound of Proposition~\ref{prop:red}; cf.\ Remark~\ref{rem:meiburg}.
\end{remark}

\begin{remark}
The asymmetry has the same source as the null-ness of $D_{s,t}$ in
Theorem~\ref{thm:dim}, and Lemma~\ref{lem:Dcover} says how much divergence costs: for
every $\varepsilon>0$ it forces infinitely many reduced approximations with
\[
  \Bigl|x-\frac pq\Bigr|\;<\;\Lambda_\tau^{\,a}\,q^{-\tau}(\log\log q)^{a},
  \qquad a=\frac{1+\varepsilon}{t},
\]
and that inclusion is what gives the zero-measure conclusion for $\kappa<-1$. This is
much weaker than requiring $\rho_k$ to stay bounded: the member $x_1$ of
Proposition~\ref{prop:sharp} diverges although
$\rho_k(x_1)=\frac{q_k^{\tau-1}}{q_{k+1}}\bigl(q_{k+1}\normx{q_kx_1}\bigr)
\sim(k+1)^{1/t}\to\infty$. Convergence, by contrast, only requires the approximations to
be eventually a little worse, which costs nothing. Theorem~\ref{thm:gaugefibre} makes the informal statement
``convergence is the generic behaviour'' quantitative inside the critical fibre, where
Corollary~\ref{thm:fibre} shows that dimension alone cannot see it.
\end{remark}

\begin{remark}
The exponent $2a/\tau$ carried by the factor $(\log\log q)^{2a/\tau}$ in the proof of
Theorem~\ref{thm:gaugeD} is irrelevant to the convergence of the series, which is why
the threshold comes out at $\kappa=-1$ for every $t>1$ and every $\tau>2$: the value of
$t$ influences only how fast the $\log\log$ factor grows. This is what the doubly
exponential growth \eqref{eq:jump} along $I$ buys; using only the Fibonacci bound
$q_k\ge\varphi^{k-1}$, which is what one gets without discarding the indices outside
$I$, would replace $\log\log q_k$ by $\log q_k$ and leave a window
$-1-\frac{2}{t\tau}\le\kappa<-1$ undecided.
\end{remark}

\section{Numerical observations on \texorpdfstring{$\pi$}{pi}}\label{sec:num}

This section reports computations over a finite range; Remark~\ref{rem:finite} below
delimits what they can mean. All values were obtained in exact rational or
high-precision ($200$-digit) arithmetic. Let $Q_k$ denote the convergent denominators
of $x=1/\pi$; these are the numerators of the convergents of $\pi$,
\[
  Q_0,Q_1,\dots \;=\; 1,\,3,\,22,\,333,\,355,\,103993,\,104348,\,208341,\,312689,\dots
\]
and by \eqref{eq:cf1} the terms of $S$ at $n=Q_k$ should be close to
$Q_{k+1}^{2}/(\pi^{2}Q_k^{3})$. Table~\ref{tab:spikes} confirms this, with agreement
improving as $a_{k+1}$ grows; the constant in \eqref{eq:cf1} degrades by at most
$(1+Q_k/Q_{k+1})^{2}\le4$ when $a_{k+1}=1$.

\begin{table}[htbp]
\centering
\small
\begin{tabular}{rrrrrr}
\toprule
$k$ & $Q_k$ & $a_{k+1}$ & $1/(Q_k^{3}\sin^{2}Q_k)$ & $Q_{k+1}^{2}/(\pi^{2}Q_k^{3})$ & ratio\\
\midrule
$0$ & $1$      & $3$   & $1.4122829$              & $0.91189065$             & $1.5487$\\
$1$ & $3$      & $7$   & $1.8597692$              & $1.8162760$              & $1.0239$\\
$2$ & $22$     & $15$  & $1.1987177$              & $1.0551657$              & $1.1360$\\
$3$ & $333$    & $1$   & $3.4802889\cdot10^{-4}$  & $3.4579942\cdot10^{-4}$  & $1.0064$\\
$4$ & $355$    & $292$ & $24.598181$              & $24.491953$              & $1.0043$\\
$5$ & $103993$ & $1$   & $2.4298972\cdot10^{-6}$  & $9.8097105\cdot10^{-7}$  & $2.4770$\\
$6$ & $104348$ & $1$   & $7.2539945\cdot10^{-6}$  & $3.8707699\cdot10^{-6}$  & $1.8740$\\
$7$ & $208341$ & $1$   & $1.6794679\cdot10^{-6}$  & $1.0954719\cdot10^{-6}$  & $1.5331$\\
$8$ & $312689$ & $2$   & $3.8873777\cdot10^{-6}$  & $2.3035723\cdot10^{-6}$  & $1.6875$\\
\bottomrule
\end{tabular}
\caption{Terms of the Flint Hills series at $n=Q_k$ against the prediction of
Proposition~\ref{prop:red}.}
\label{tab:spikes}
\end{table}

The cluster effect of Remark~\ref{rem:cluster} is equally accurate over the finite
range in which \eqref{eq:clusterF} applies.

\begin{table}[htbp]
\centering
\small
\begin{tabular}{rll}
\toprule
$Q_k$ & $\sum_{m=1}^{59}\bigl((mQ_k)^{3}\sin^{2}(mQ_k)\bigr)^{-1}$
& prediction $\zeta(5)/(Q_k^{3}\sin^{2}Q_k)$ \\
\midrule
$22$  & $1.2429888$  & $1.2429837$ \\
$355$ & $25.5065363$ & $25.5065368$ \\
\bottomrule
\end{tabular}
\caption{The multiples of a convergent denominator contribute a factor close to
$\zeta(5)=1.0369278\ldots$ over the range $m\le59$.}
\label{tab:cluster}
\end{table}

Numerically $\sum_{n\le10^{7}}(n^{3}\sin^{2}n)^{-1}=30.3145459\ldots$, and the four
terms $n=1,3,22,355$ already account for $29.0690$ of it. Nothing later in that range is
comparable: $n=103993$ has $\abs{\sin n}\approx1.9\cdot10^{-5}$, but
$n^{3}\approx1.1\cdot10^{15}$ overwhelms it.

Table~\ref{tab:cf} summarizes the first $8000$ partial quotients of $\pi$, computed from
an $11000$-digit rational approximation; recomputing from a $12500$-digit approximation
reproduces all $8101$ computed partial quotients, so the range used is stable. Here
$q_k$ denotes the convergent denominators of $\pi$ itself, not of $x=1/\pi$. The two
indexings are related by a shift: $1/\pi=[0;3,7,15,\dots]$ is the shift of
$\pi=[3;7,15,\dots]$, so numerators and denominators are interchanged with an index
displacement,
\[
  Q_{k+1}=p_k(\pi),\qquad P_{k+1}=q_k(\pi)\qquad(k\ge0),\qquad Q_0=1 ,
\]
and $p_k(\pi)\asymp\pi q_k(\pi)$. Since $\mu(1/\pi)=\mu(\pi)$
(Lemma~\ref{lem:inv}) and the partial quotients are literally the same sequence, the
exponents below are the same in either indexing. Note also that in \eqref{eq:mu} the
partial quotient $a_{k+1}$ is weighed against $q_k$, not against $q_{k+1}$.

\begin{table}[htbp]
\centering
\footnotesize
\begin{tabular}{@{}lll@{}}
\toprule
quantity & value for $\pi$ & almost every $x$ \\
\midrule
largest partial quotient, $1\le k\le8000$ & $a_{431}=20776$ & --- \\
\quad the denominator it is weighed against & $\log_{10}q_{430}=215.7820$ & --- \\
\quad its own contribution to \eqref{eq:mu} & $2+\log a_{431}/\log q_{430}=2.0200$ & --- \\
proportion of $a_k=1$, $1\le k\le8000$ & $0.420750$ & $\log_2(4/3)=0.415037$ \\
$q_{8000}^{1/8000}$ & $3.2449259$ & $e^{\pi^{2}/(12\log2)}=3.2758229$ \\
$\max_{1\le k\le7999}\bigl(2+\frac{\log a_{k+1}}{\log q_k}\bigr)$
   & $3.3916625$ \ ($k=1$; $q_1=7$, $a_2=15$) & $\limsup=2$ \\
\quad the same for $2\le k\le7999$ & $3.2008225$ \ ($k=3$; $q_3=113$, $a_4=292$) & \\
\quad the same for $100\le k\le7999$ & $2.0270886$ \ ($k=117$) & \\
$\max_{100\le k\le7999}\bigl(\log a_{k+1}\big/\tfrac12\log q_k\bigr)$
   & $0.0541772$ & $\limsup=0$ \\
\bottomrule
\end{tabular}
\caption{Continued fraction statistics of $\pi$. The third column gives the
almost-everywhere values: the Gauss--Kuzmin law for the proportion of $1$'s, L\'evy's
theorem for $q_k^{1/k}$, and $\mu=2$ for the last three rows. A maximum taken over a
range with a fixed lower endpoint is non-decreasing in the upper endpoint, so these are
finite-range observations, not approximations to a limit.}
\label{tab:cf}
\end{table}

\paragraph{The correct necessary condition.}
It is tempting to read the last row of Table~\ref{tab:cf} as measuring a necessary
condition $a_{k+1}\gtrsim q_k^{1/2}$ for the divergence of $S$. No such condition
follows: that reading confuses divergence caused by terms not tending to $0$ with
divergence by accumulation of small positive terms. What Proposition~\ref{prop:red} gives
is
\begin{equation}\label{eq:nec}
  \mathcal F_{3,2}(x)=\infty
  \quad\Longrightarrow\quad
  \limsup_{k\to\infty}\frac{\log a_{k+1}}{\tfrac12\log q_k}\ \ge\ 1,
\end{equation}
equivalently: for every $\varepsilon>0$ one has $a_{k+1}>q_k^{1/2-\varepsilon}$ for
infinitely many $k$. Indeed, if $a_{k+1}\le q_k^{(1-\delta)/2}$ for all large $k$ and
some $\delta>0$, then $q_{k+1}\le2a_{k+1}q_k\le2q_k^{(3-\delta)/2}$, so
$q_{k+1}^{2}/q_k^{3}\le4q_k^{-\delta}$, which is summable by \eqref{eq:cf3}. The
condition \eqref{eq:nec} permits the ratio to approach $1$ from below: the member
$x_1$ of Proposition~\ref{prop:sharp} (with $t=2$, $\beta=3/2$, $r=1$, i.e.\
$a_{k+1}=\bigl\lceil\sqrt{q_k/(k+1)}\bigr\rceil$) satisfies
\[
  \frac{q_{k+1}^{2}}{q_k^{3}}\sim\frac{1}{k+1},\qquad
  \frac{a_{k+1}}{\sqrt{q_k}}\sim\frac{1}{\sqrt{k+1}}\longrightarrow0,
\]
and diverges. The last row of Table~\ref{tab:cf} should therefore be read as a
measurement of the ratio in \eqref{eq:nec}, whose limit superior would have to reach
$1$; over the computed range it stays below $0.055$ for $k\ge100$. The record partial
quotient $a_{431}=20776$ is no exception: it is to be compared with $q_{430}$, and
$q_{430}^{1/2}\approx10^{107.9}$, so the ratio in \eqref{eq:nec} equals $0.0400$ there.

\begin{remark}\label{rem:finite}
No finite portion of the continued fraction of $\pi$ can bear on the convergence of
$S$. Given any finite string of partial quotients, the tail can be continued so as to
make the resulting number converge (Proposition~\ref{prop:sharp} with $r=2$, after the
prescribed initial string) or diverge (any $r\le1$). The content of
Section~\ref{sec:num} is therefore a verification of Proposition~\ref{prop:red} and of
the classical statistics of $\pi$'s expansion over the computed range, and an
observation that no anomaly appears there --- not evidence about the limit superior in
\eqref{eq:nec}.
\end{remark}

\section{Discussion}\label{sec:disc}

\emph{What is decidable now.} For $x=1/\pi$, every member of the family outside the
window $t<s\le t(M-1)$, where $M=7.103205334137\ldots$ is the current upper bound for
$\mu(\pi)$ --- for $t=2$, everything outside $2<s\le12.20642$ ---
together with the dimension of the divergence set (Theorem~\ref{thm:dim}), the structure
of the critical case (Proposition~\ref{prop:sharp}, Corollary~\ref{cor:noiff}), the
dimensions of both parts of the critical fibre (Corollary~\ref{thm:fibre}) and their
$h_\kappa$-measures for every $\kappa$ (Theorems~\ref{thm:gaugeD} and
\ref{thm:gaugefibre}). Each of these follows from Proposition~\ref{prop:red} combined
with classical Diophantine approximation.

\emph{What is not.} The convergence of $S$ itself, which by Corollary~\ref{cor:pi} sits
inside that window and is equivalent to a statement about the single number $\pi$.
The known bounds on $\mu(\pi)$ come from Pad\'e and hypergeometric constructions
\cite{Hata,Salikhov,ZeilbergerZudilin} and have moved from $8.0161$ to $7.1032$ in
thirty years; we are not aware of a variant of these constructions that is expected to
reach $5/2$, though we know of no proof that they cannot. Separately, the metric theory
of Section~\ref{sec:dim} is about the size of exceptional sets, and a size estimate does
not determine whether a prescribed number lies in the set.

\emph{Where there is room.} Three directions, none of which requires an improvement of
$\mu(\pi)$.

\begin{enumerate}
\item \emph{Beyond the gauge scale.} Theorem~\ref{thm:gaugeD} settles the
$h_\kappa$-measure of $D_{s,t}$ for every $\kappa$, so the family \eqref{eq:gauge} is
exhausted. The next scale is the one it cannot resolve: the exact-order function of
$x$ itself. By Lemma~\ref{lem:bridge} the fate of $\Fst(x)$ is decided by the growth of
$\rho_k=q_k^{\tau}\abs{x-p_k/q_k}$, and Theorems~\ref{thm:gaugeD}
and~\ref{thm:gaugefibre} show that $C_\tau$ and $E_\tau$ are both $h_\kappa$-large for
$\kappa\ge-1$. A gauge theory indexed by $\psi$ rather than by a
power-times-logarithm --- a Hausdorff-measure analogue of Theorem~\ref{thm:bugeaud},
which \cite{BakerWard} provides only at the level of dimension --- would be the natural
next invariant (Question~\ref{q:fibre}).

\item \emph{Quantitative metric statements.} The quantity
$\Sigma_{s,t}(x)=\sum_kq_{k+1}^{t}/q_k^{s}$ is a weighted sum along the orbit of the
Gauss map $T(x)=\{1/x\}$, whose transfer operator is quasi-compact with spectral gap
governed by the Gauss--Kuzmin--Wirsing constant $0.3036\ldots$. It should be stressed
that the spectral gap alone does not yield distributional or tail estimates for
$\Sigma_{s,t}$: the summands are unbounded and not integrable against the Gauss measure,
so any such programme must first fix a function space and a weight for which the
relevant observable is admissible. Making this precise --- the law of $\Sigma_{s,t}$,
the tail $\mathbb P\{\Sigma_{s,t}>\lambda\}$, and multifractal refinements of
\eqref{eq:JB} for the level sets of $\limsup\log q_{k+1}/\log q_k$ --- seems to us the
most promising technical direction.

\item \emph{Effectivity.} The operative question for $\pi$ is not the value of
$\mu(\pi)$ but whether an effective form of ``almost every $x$ satisfies
$\Sigma_{3,2}(x)<\infty$'' can be attached to a prescribed constant. Results relating
irrationality exponents to effective Hausdorff dimension \cite{Effectivization}
indicate the shape such a statement would take, though not how to prove one for $\pi$.
\end{enumerate}

\begin{question}\label{q:fibre}
Theorem~\ref{thm:gaugefibre} shows that the gauges \eqref{eq:gauge} separate $C_\tau$
from $E_\tau$ exactly for $\kappa<-1$, and that both are $h_\kappa$-infinite for
$\kappa\ge-1$. Is there a dimension function $h$, not of the
form \eqref{eq:gauge}, for which $\mathcal H^{h}$ distinguishes the sets
$E(\psi)$ of Theorem~\ref{thm:bugeaud} for order functions $\psi$ with the same
exponent --- e.g.\ $q^{-\tau}$ from $q^{-\tau}(\log q)^{\gamma}$? This would be a
Hausdorff-measure refinement of Theorem~\ref{thm:bugeaud}; \cite{BakerWard} computes
dimensions only.
\end{question}

\begin{question}
Proposition~\ref{prop:red} reduces the convergence of $S$ to
$\sum_kq_{k+1}^{2}/q_k^{3}<\infty$. For varying irrational $x$, Corollary~\ref{cor:noiff}
shows that this summability criterion is strictly weaker than $\mu(x)<5/2$; returning to
$x=1/\pi$, is there an approach to the weaker statement that does not
proceed through a bound on $\mu(\pi)$? Note that a bound $a_{k+1}(\pi)\le Cq_k(\pi)^{\theta}$
for all large $k$ would give $\mu(\pi)\le2+\theta$ by \eqref{eq:mu}, so any route
through uniform bounds on the partial quotients is an improvement of $\mu(\pi)$ and is
not a way around it. Conversely, $\mu(\pi)\le M$ yields only
$a_{k+1}\le q_k^{M-2+\varepsilon}$ for large $k$, for each $\varepsilon>0$, and not a
bound valid for all $k$ at the endpoint exponent.
\end{question}

\begin{question}
Does \cite[Conj.~3.1]{Meiburg} hold in the range $0<v\le1<u$? (By
Remark~\ref{rem:meiburg} the conjecture fails for $v<u\le1$, so this is what remains of
it.) By that remark this requires a substitute for Proposition~\ref{prop:red} in which
the dyadic blocks are not dominated by the spikes.
\end{question}

\paragraph{Provenance.} Proposition~\ref{prop:red} is classical in substance; the
explicit two-sided form is stated here for convenience.
Corollary~\ref{cor:crit}(a) is \cite[Thm.~2.5]{Meiburg}; Alekseyev's sufficient
condition \cite[Thm.~5]{Alekseyev} is the weaker $\mu<1+(s-1)/t$, and his
\cite[Cor.~4]{Alekseyev} is the necessary direction, i.e.\ the contrapositive of
Corollary~\ref{cor:crit}(b) at $(s,t)=(3,2)$. Corollary~\ref{cor:pi} is a specialization
of these. Theorem~\ref{thm:dim} is an immediate combination of that criterion with the
Jarn\'ik--Besicovitch theorem; the author is not aware of it being recorded in the Flint
Hills literature, but it should be regarded as folklore rather than new. In
Proposition~\ref{prop:sharp}, for $u>v>1$ the member $x_0$ recovers both nonconvergence
conclusions of \cite[Thm.~3.1]{Meiburg}, while the members with $r>1$ establish
\cite[Conj.~3.1]{Meiburg} in that parameter range; Corollary~\ref{cor:noiff} is a consequence. Lemma~\ref{lem:bridge} puts the
summability criterion into the form used throughout the rest of the note. In
Section~\ref{sec:gauge}, Lemma~\ref{lem:HC} is the classical Hausdorff--Cantelli
argument and Theorem~\ref{thm:jarnik} is Jarn\'ik's theorem in the form given in
\cite[\S1.2]{BeresnevichVelani} and \cite{BDV}; Lemma~\ref{lem:Dcover} and
Theorems~\ref{thm:gaugeD} and~\ref{thm:gaugefibre} appear to be new.
Corollary~\ref{thm:fibre} is immediate from \cite{BugeaudExact} as well
(Remark~\ref{rem:viaexact}) and should be regarded as folklore. Theorem~\ref{thm:bugeaud} is used only in Remark~\ref{rem:viaexact}, which
gives an alternative derivation of Corollary~\ref{thm:fibre}; it is quoted in the
formulation of \cite[\S1.3]{FraserWheeler}, the original \cite{BugeaudExact} not having
been available to the author.

\paragraph{AI assistance.}
\begin{sloppypar}
Claude (Anthropic) was used in drafting and in the numerical
work of Section~\ref{sec:num}. ChatGPT (OpenAI) was used for mathematical review and
suggested the continued-fraction construction \eqref{eq:constr} of
Proposition~\ref{prop:sharp} and the sparse-index argument underlying
Lemma~\ref{lem:Dcover} and Theorems~\ref{thm:gaugeD}
and~\ref{thm:gaugefibre}. The same
review identified the missing reference \cite{Meiburg} and the resulting
misattributions, corrected the necessary condition discussed after Table~\ref{tab:cf},
and removed an unproved dimension claim about the critical fibre. The author is
responsible for the statements, the proofs, the numerical results and the references.
\end{sloppypar}

\end{document}